\documentclass[11pt,twoside]{article}
\usepackage{amssymb,mathrsfs,cite,color}
\usepackage{amsfonts,amsmath,latexsym}

\usepackage{indentfirst}
\usepackage[amsmath,thmmarks]{ntheorem}
\numberwithin{equation}{section}
\theoremheaderfont{\normalfont\rmfamily\bf}
\theorembodyfont{\normalfont\it }
\newtheorem{theorem}{\indent Theorem}[section]
\newtheorem{lemma}{\indent Lemma} [section]
\newtheorem{corollary}{\indent Corollary} [section]

\newtheorem{definition}{\indent Definition} [section]
\newtheorem{remark}{\indent Remark} [section]

  \theoremstyle{nonumberplain}

\theorembodyfont{\normalfont \rm }
 \theoremsymbol{$\square$}
  \newtheorem{proof}{\indent Proof}

\allowdisplaybreaks

\DeclareMathOperator*{\essinf}{ess\, inf}
\DeclareMathOperator*{\esssup}{ess\, sup}
\def\rr{\mathbb{R}}
\def\rn{{\rr}^n}
\def\loc{{\rm loc}}

\begin{document}

\title{\bf Characterization of Lipschitz functions in terms of variable
exponent Lebesgue spaces \footnote{Supported by National Natural
Science Foundation of China (Grant Nos. 11571160 and 11471176).}
  }
\author{Pu Zhang   
         \\
 {\small\it Department of Mathematics, Mudanjiang Normal
University, Mudanjiang 157011, P. R. China}\\
 {\small\it  E-mail:  puzhang@sohu.com 
               }
    }
\date{ }
\maketitle

\begin{center}\begin{minipage}{14cm}

{\bf Abstract}~  Our aim is to characterize
the Lipschitz functions by variable exponent Lebesgue spaces.
We give some characterizations of the boundedness of the maximal
or nonlinear commutators of the Hardy-Littlewood maximal function
and sharp maximal function in variable exponent Lebesgue spaces
when the symbols $b$ belong to the Lipschitz spaces,
by which some new characterizations of Lipschitz spaces and
nonnegative Lipschitz functions are obtained.
Some equivalent relations between the Lipschitz norm and the
variable exponent Lebesgue norm are also given.


{\bf MR(2010) Subject Classification}~ 47B47, 42B25, 46E30, 42B20, 42B35, 26A16

\end{minipage}
\end{center}

\medskip

\section{Introduction and Results}

Let $T$ be the  classical singular integral operator, the commutator
$[b,T]$ generated by $T$ and a suitable function $b$ is given by
\begin{equation} \label{equ.1.1}    
[b,T](f)(x) = T\big((b(x)-b)f\big)(x) = b(x)T(f)(x)-T(bf)(x).
\end{equation}

A well known result due to Coifman, Rochberg and Weiss \cite{crw}
states that $[b,T]$ is bounded on $L^p(\rn)$ for $1<p<\infty$
when $b\in {BMO(\rn)}$. They also gave a characterization of $BMO(\rn)$
in virtue of the $L^p-$boundedness of the above commutator.
In 1978, Janson \cite{j} gave some characterizations of Lipschitz
space ${\dot{\Lambda}}_{\beta}(\rn)$ via commutator $[b,T]$ and proved
that $b\in {\dot{\Lambda}_{\beta}(\rn)} (0<\beta<1)$ if and only if
$[b,T]$ is bounded from $L^p(\rn)$ to $L^q(\rn)$ where $1<p<n/\beta$
and $1/p-1/q=\beta/n$ (see also Paluszy\'nski \cite{p}).

As usual, a cube $Q\subset \rn$ always means its sides parallel
to the coordinate axes. Denote by $|Q|$ the Lebesgue measure of $Q$
and $\chi_Q$ the characteristic function of $Q$. For
$f\in {L}_{\loc}^{1}(\rn)$, we write $f_Q={|Q|}^{-1}\int_Qf(x)dx$.
The Hardy-Littlewood maximal function $M$ is defined by
$$M(f)(x)=\sup_{Q\ni x} \frac{1}{|Q|} \int_{Q} |f(y)| dy,
$$
and the sharp maximal function $M^{\sharp}$ is defined by
$$M^{\sharp}f(x)=\sup_{Q\ni x} \frac{1}{|Q|} \int_{Q}
|f(y)-f_{Q}|{d}y,
$$
where the supremum is taken over all cubes $Q\subset \rn$
containing $x$.

The maximal commutator of $M$ with a locally integrable
function $b$ is defined by
$$M_b(f)(x)= M\big((b(x)-b)f\big)(x)
=\sup_{Q\ni x} \frac{1}{|Q|} \int_{Q} |b(x)-b(y)||f(y)|dy,
$$
where the supremum is taken over all cubes $Q\subset\rn$ containing
$x$.

The boundedness of the maximal commutator $M_b$ have been
studied intensively by many authors when the symbol $b$ belongs
to $BMO(\rn)$. See \cite{agkm, ghst, hly, hy, st1, zhw2} for instance.

Similar to (\ref{equ.1.1}), we can also define the (nonlinear)
commutators of $M$ and $M^{\sharp}$ with a locally
integrable function $b$ by
$$[b,M](f)=bM(f) -M(bf) \quad \hbox{and} \quad
[b,M^{\sharp}](f)= bM^{\sharp}(f) -M^{\sharp}(bf).
$$

We would like to remark that operators $M_b$ and $[b,M]$ essentially
differ from each other. For example, $M_b$ is positive and
sublinear, but $[b,M]$ is neither positive nor sublinear.

Using real interpolation techniques, Milman and Schonbek \cite{ms}
established a commutator result, by which they obtained
the $L^p-$boundedness of $[b,M]$ when $b \in {BMO(\rn)}$ and $b\ge 0$.
This operator can be used in studying the product of a function
in $H^1$ and a function in $BMO(\rn)$ (see \cite{bijz} for instance).
When the symbol $b$ belongs to $BMO(\rn)$,
Bastero, Milman and Ruiz \cite{bmr} characterized the
boundedness of $[b,M]$ and
$[b,M^{\sharp}]$ on $L^p$ spaces when $1<p<\infty$.
Zhang and Wu \cite{zhw1} obtained similar results for
the fractional maximal function. Agcayazi et al. \cite{agkm}
gave the end-point estimates for $[b,M]$.
In 2015, the author extended some of these results
to the multilinear setting in \cite{zh1}.

Zhang and Wu \cite{zhw3} studied the boundedness
of $[b,M]$ and $[b,M^{\sharp}]$ on variable exponent Lebesgue
spaces when $b\in{BMO(\rn)}$ and extended the results to the
fractional case in \cite{zhw2}.
Recently, the author \cite{zh2} gave some characterizations
for the boundedness of $M_b$ and $[b,M]$ on Lebesgue
and Morrey spaces when the symbols $b$ belong to Lipschitz
space, by which some new characterizations of Lipschitz and
nonnegative Lipschitz functions are given.

On the other hand, Ho \cite{ho} obtained some characterizations
of the $BMO$ and Lipschitz spaces by the norm of
rearrangement-invariant Banach function spaces. Izuki and Sawano
\cite{is} gave a characterization of $BMO(\rn)$ by using the norm
of variable exponent Lebesgue spaces.

Motivated by \cite{bmr}, \cite{ho}, \cite{is} and \cite{zh2},
we will study the characterization of Lipschitz
functions in the context of variable exponent Lebesgue spaces.
Firstly, we prove the boundedness of $M_b$, $[b,M]$ and $[b,M^{\sharp}]$
in variable exponent Lebesgue spaces when the symbols $b$ belong
to Lipschitz space, by which some new characterizations of
Lipschitz spaces and nonnegative Lipschitz functions are obtained.
Next, we give some equivalent relations between the
Lipschitz norm and the variable exponent Lebesgue norm.

To state our result, we first recall some notation and definitions.

\begin{definition}   \label{lip}
Let $0<\beta<1$, we say a function $b$ belongs to the Lipschitz
space $\dot{\Lambda}_{\beta}(\rn)$ if there exists a constant $C$
such that for all $x,y\in\rn$,
$$|b(x)-b(y)|\le {C}|x-y|^{\beta}.
$$
The smallest such constant $C$ is called the $\dot{\Lambda}_{\beta}$
norm of $b$ and is denoted by $\|b\|_{\dot{\Lambda}_{\beta}}$.
\end{definition}

\begin{definition} \label{def.varLp}
Let $p(\cdot): \rn \to [1,\infty)$ be a measurable function. The
variable exponent Lebesgue space, $L^{p(\cdot)}(\rn)$, is defined
by
$$ L^{p(\cdot)}(\rn)=\bigg\{f~ \mbox{measurable}:
\int_{\rn} \bigg(\frac{|f(x)|}{\lambda}\bigg)^{p(x)} dx <\infty
~\mbox{for some constant}~ \lambda>0\bigg\}.
$$
\end{definition}
It is known that the set $L^{p(\cdot)}(\rn)$ becomes a
Banach space with respect to the norm
\begin{equation*}
\|f\|_{L^{p(\cdot)}(\rn)}=\inf \bigg\{ \lambda>0: \int_{\rn}
\bigg(\frac{|f(x)|}{\lambda} \bigg)^{p(x)} dx \le 1 \bigg\}.
\end{equation*}

The readers are referred to \cite{cf} and \cite{dhhr} for some
properties and applications of $L^{p(\cdot)}(\rn)$.

Denote by $\mathscr{P}(\rn)$ the set of all measurable functions
$p(\cdot): \rn\to[1,\infty)$ such that
$$1< p_{-}:=\essinf_{x\in \rn}p(x) ~~\mathrm{and}~~
{p_{+}:}=\esssup_{x\in \rn}p(x)<\infty,
$$
and by $\mathscr{B}(\rn)$ the set of all $p(\cdot) \in
\mathscr{P}(\rn)$ such that $M$ is bounded on
$L^{p(\cdot)}(\rn)$.

\begin{remark} \label{remark1}
For any $p(\cdot)\in \mathscr{B}(\rn)$ and $\lambda>1$, then,
by Jensen's inequality, we have ${\lambda} p(\cdot)\in \mathscr{B}(\rn)$.
See Remark 2.13 in \cite{cfmp}.
\end{remark}

For convenience, we introduce a notation $\mathscr{B}_{p,q}^{\alpha}(\rn)$
as follows.

\begin{definition}
We say an ordered pair of variable exponents
$(p(\cdot),q(\cdot)) \in \mathscr{B}_{p,q}^{\alpha}(\rn)$,
if $p(\cdot)\in \mathscr{P}(\rn)$, $0<{\alpha}<n/p_{+}$ and
$1/q(\cdot)=1/p(\cdot)-\alpha/n$ with $q(\cdot)(n-\alpha)/n \in
\mathscr{B}(\rn)$.
\end{definition}

\begin{remark} \label{remark2}
The condition $q(\cdot)(n-\alpha)/n \in \mathscr{B}(\rn)$ is
equivalent to saying that
there exists $q_0$ with $n/(n-\alpha)<q_0<\infty$
such that $q(\cdot)/q_0 \in \mathscr{B}(\rn)$. Moreover,
$q(\cdot)(n-\alpha)/n \in \mathscr{B}(\rn)$ implies
$q(\cdot) \in \mathscr{B}(\rn)$ by Remark \ref{remark1}.
See Remark 2.13 in \cite{cfmp} for details.
\end{remark}

Our results can be stated as follows.

\begin{theorem} \label{thm.mc-lp} 
Let $b$ be a locally integrable function and $0<\beta<1$,
then the following assertions are equivalent:

(1)\ \ $b\in {\dot{\Lambda}_{\beta}(\rn)}$.

(2)\ \ $M_b$ is bounded from $L^{p(\cdot)}(\rn)$ to $L^{q(\cdot)}(\rn)$
for all $(p(\cdot),q(\cdot)) \in \mathscr{B}_{p,q}^{\beta}(\rn)$.

(3)\ \ $M_b$ is bounded from $L^{p(\cdot)}(\rn)$ to
$L^{q(\cdot)}(\rn)$  for some $(p(\cdot),q(\cdot)) \in
\mathscr{B}_{p,q}^{\beta}(\rn)$.

(4)\ \  There exists some $(p(\cdot),q(\cdot)) \in
\mathscr{B}_{p,q}^{\beta}(\rn)$, such that
\begin{equation}     \label{equ.1.2}
\sup_{Q} \frac1{|Q|^{\beta/n}}
\frac{\|(b-b_{Q})\chi_{Q}\|_{L^{q(\cdot)}(\rn)}}
{\|\chi_{Q}\|_{L^{q(\cdot)}(\rn)}} <\infty.
\end{equation}

(5)\ \ For all $(p(\cdot),q(\cdot)) \in
\mathscr{B}_{p,q}^{\beta}(\rn)$, we have
$$\sup_{Q} \frac1{|Q|^{\beta/n}} \frac{\|(b- b_{Q})\chi_{Q}\|_{L^{q(\cdot)}(\rn)}}
{\|\chi_{Q}\|_{L^{q(\cdot)}(\rn)}} <\infty.
$$
\end{theorem}

For a fixed cube $Q_0$, the Hardy-Littlewood maximal function
related to $Q_0$ is given by
$$M_{Q_0}(f)(x) =\sup_{Q_0\supseteq{Q}\ni{x}}
\frac{1}{|Q|}\int_{Q}|f(y)|dy,
$$
where the supremum is taken over all the cubes $Q$ with $Q\subseteq
{Q_0}$ and $Q\ni {x}$.

For the commutator $[b,M]$, we have the following result.

\begin{theorem} \label{thm.nc-lp}
Let $b$ be a locally integrable function and $0<\beta<1$, then the
following assertions are equivalent:

(1)\ \ $b\in \dot{\Lambda}_{\beta}(\rn)$ and $b\ge 0$.

(2)\ \  $[b,M]$ is bounded from $L^{p(\cdot)}(\rn)$ to
$L^{q(\cdot)}(\rn)$ for all $(p(\cdot),q(\cdot)) \in
\mathscr{B}_{p,q}^{\beta}(\rn)$.

(3)\ \  $[b,M]$ is bounded from $L^{p(\cdot)}(\rn)$ to
$L^{q(\cdot)}(\rn)$ for some $(p(\cdot),q(\cdot)) \in
\mathscr{B}_{p,q}^{\beta}(\rn)$.

(4)\ \  There exists some $(p(\cdot),q(\cdot)) \in
\mathscr{B}_{p,q}^{\beta}(\rn)$, such that
$$\sup_{Q} \frac1{|Q|^{\beta/n}}
\frac{\|(b-M_{Q}(b))\chi_{Q}\|_{L^{q(\cdot)}(\rn)}}
{\|\chi_{Q}\|_{L^{q(\cdot)}(\rn)}}
<\infty.
$$

(5)\ \ For all $(p(\cdot),q(\cdot)) \in
\mathscr{B}_{p,q}^{\beta}(\rn)$, we have
$$\sup_{Q} \frac1{|Q|^{\beta/n}}
\frac{\|(b- M_{Q}(b))\chi_{Q}\|_{L^{q(\cdot)}(\rn)}}
{\|\chi_{Q}\|_{L^{q(\cdot)}(\rn)}} <\infty.
$$
\end{theorem}

For the commutator $[b,M^{\sharp}]$, we have the following result.

\begin{theorem}\label{thm.nc-sharp}
Let $b$ be a locally integrable function and $0<\beta<1$, then the
following assertions are equivalent:

(1)\ \ $b\in \dot{\Lambda}_{\beta}(\rn)$ and $b\ge 0$.

(2)\ \  $[b,M^{\sharp}]$ is bounded from $L^{p(\cdot)}(\rn)$ to
$L^{q(\cdot)}(\rn)$ for all $(p(\cdot),q(\cdot)) \in
\mathscr{B}_{p,q}^{\beta}(\rn)$.

(3)\ \  $[b,M^{\sharp}]$ is bounded from $L^{p(\cdot)}(\rn)$ to
$L^{q(\cdot)}(\rn)$ for some $(p(\cdot),q(\cdot)) \in
\mathscr{B}_{p,q}^{\beta}(\rn)$.

(4)\ \  There exists some $(p(\cdot),q(\cdot)) \in
\mathscr{B}_{p,q}^{\beta}(\rn)$, such that
$$\sup_{Q} \frac1{|Q|^{\beta/n}}
\frac{\|(b-2M^{\sharp}(b\chi_Q))\chi_Q\|_{L^{q(\cdot)}(\rn)}}
{\|\chi_Q\|_{L^{q(\cdot)} (\rn)}} <\infty.
$$

(5)\ \ For all $(p(\cdot),q(\cdot)) \in
\mathscr{B}_{p,q}^{\beta}(\rn)$, we have
$$\sup_{Q} \frac1{|Q|^{\beta/n}}
\frac{\|(b-2M^{\sharp}(b\chi_Q))\chi_Q\|_{L^{q(\cdot)}(\rn)}}
{\|\chi_Q\|_{L^{q(\cdot)} (\rn)}}
<\infty.
$$
\end{theorem}

\begin{remark}
Results similar to the ones of
Theorems \ref{thm.mc-lp}\---\ref{thm.nc-sharp}
were established in \cite{zhw3} when $b\in{BMO(\rn)}$,
which may be viewed as the end-point case of $\beta=0$.
When $p$ and $q$ are constants, Theorems \ref{thm.mc-lp}
and \ref{thm.nc-lp} were obtained in \cite{zh2}.
\end{remark}

\begin{remark}
We would like to note that some idea of the proof
of Theorems \ref{thm.nc-lp} and \ref{thm.nc-sharp} comes
from \cite{bmr} and \cite{zh2}.
\end{remark}

Next, we characterize the Lipschitz functions by using the norm
of variable exponent Lebesgue spaces. We introduce the following
three classes of functions for convenience.

\begin{definition}       \label{def.1.4}
Let $0<\beta<1$ and $q(\cdot)\in \mathscr{B}(\rn)$. We define
$${\dot{\Lambda}_{\beta,q(\cdot)}(\rn)} =\bigg\{b\in{L^1_{\rm{loc}}(\rn)}:
\|b\|_{\dot{\Lambda}_{\beta,q(\cdot)}} =
\sup_{Q} \frac1{|Q|^{\beta/n}} \frac{\|(b- b_{Q})\chi_{Q}\|_{L^{q(\cdot)}(\rn)}}
{\|\chi_{Q}\|_{L^{q(\cdot)}(\rn)}} <\infty \bigg\},
$$
$${\dot{\Lambda}_{\beta,q(\cdot)}^{\ast}(\rn)} =\bigg\{b\in{L^1_{\rm{loc}}(\rn)}:
\|b\|_{\dot{\Lambda}_{\beta,q(\cdot)}^{\ast}}= \sup_{Q} \frac1{|Q|^{\beta/n}}
\frac{\|(b- M_{Q}(b))\chi_{Q}\|_{L^{q(\cdot)}(\rn)}}
{\|\chi_{Q}\|_{L^{q(\cdot)}(\rn)}} <\infty \bigg\},
$$
and
$${\dot{\Lambda}_{\beta,q(\cdot)}^{\sharp}(\rn)}=\bigg\{b\in{L^1_{\rm{loc}}(\rn)}:
\|b\|_{\dot{\Lambda}_{\beta,q(\cdot)}^{\sharp}}= \sup_{Q} \frac1{|Q|^{\beta/n}}
\frac{\|(b-2M^{\sharp}(b\chi_Q))\chi_Q\|_{L^{q(\cdot)}(\rn)}}
{\|\chi_Q\|_{L^{q(\cdot)}(\rn)}} <\infty \bigg\}.
$$
\end{definition}

\begin{theorem} \label{thm.equiv-norm-1} 
Let $0<\beta<1$. Then
${\dot{\Lambda}_{\beta}(\rn)}={\dot{\Lambda}_{\beta,q(\cdot)}(\rn)}$
 for all $q(\cdot)\in \mathscr{B}(\rn)$.
Furthermore, there exist positive constants $C_1$ and $C_2$
such that, for all $b\in {\dot{\Lambda}_{\beta}(\rn)}$,
$$C_1\|b\|_{\dot{\Lambda}_{\beta}} \le
\|b\|_{\dot{\Lambda}_{\beta,q(\cdot)}}
\le{C_2}\|b\|_{\dot{\Lambda}_{\beta}}.
$$
\end{theorem}

\begin{remark}
Similar characterization for $BMO$ functions, which can be
viewed as $\beta=0$,
is proved by Izuki in \cite{i2} (see also \cite{is}).
We would like to remark that Ho \cite{ho} obtained a
characterization of Lipschitz space with respect to
rearrangement invariant Banach spaces. However, the variable
exponent Lebsegue spaces are not rearrangement invariant.
\end{remark}

Denoted by
${\dot{\Lambda}_{\beta}^{+}(\rn)}:=\{b: 0\le{b}\in \dot{\Lambda}_{\beta}(\rn)\}$,
the set of all nonnegative Lipschitz functions. We have the following
characterizations of ${\dot{\Lambda}_{\beta}^{+}(\rn)}$ in terms of
variable exponent Lebesgue norm.

\begin{theorem} \label{thm.equiv-norm-2} 
Let $0<\beta<1$. Then
${\dot{\Lambda}_{\beta}^{+}(\rn)}={\dot{\Lambda}_{\beta,q(\cdot)}^{\ast}(\rn)}$
for all $q(\cdot)\in \mathscr{B}(\rn)$.
Furthermore, there exist positive constants $C_1$ and $C_2$
such that, for all $b\in {\dot{\Lambda}_{\beta}^{+}(\rn)}$,
$$C_1\|b\|_{\dot{\Lambda}_{\beta}} \le
\|b\|_{\dot{\Lambda}_{\beta,q(\cdot)}^{\ast}}
\le{C_2}\|b\|_{\dot{\Lambda}_{\beta}}.
$$
\end{theorem}

\begin{theorem}\label{thm.equiv-norm-3}
Let $0<\beta<1$. Then
${\dot{\Lambda}_{\beta}^{+}(\rn)}={\dot{\Lambda}_{\beta,q(\cdot)}^{\sharp}(\rn)}$
for all $q(\cdot)\in \mathscr{B}(\rn)$.
Furthermore, there exist positive constants $C_1$ and $C_2$
such that, for all $b\in {\dot{\Lambda}_{\beta}^{+}(\rn)}$,
$$C_1\|b\|_{\dot{\Lambda}_{\beta}} \le
\|b\|_{\dot{\Lambda}_{\beta,q(\cdot)}^{\sharp}}
\le{C_2}\|b\|_{\dot{\Lambda}_{\beta}}.
$$
\end{theorem}

Obviously, the ranges of $q(\cdot)$
in the fourth and fifth assertions of Theorems \ref{thm.mc-lp},
\ref{thm.nc-lp} and \ref{thm.nc-sharp} are not full since
$(p(\cdot),q(\cdot)) \in \mathscr{B}_{p,q}^{\beta}(\rn)$ implies
$n/(n-\beta)<{q(\cdot)} <\infty$.
There is a gap when $1<{q(\cdot)}\le{n/(n-\beta)}$.
We would like to remark that by Theorems \ref{thm.equiv-norm-1},
\ref{thm.equiv-norm-2} and \ref{thm.equiv-norm-3}, we can improve
the fourth and fifth assertions of Theorems \ref{thm.mc-lp},
\ref{thm.nc-lp} and \ref{thm.nc-sharp} to the full range for $q(\cdot)$.
Here, as an example, we only rewrite Theorem \ref{thm.mc-lp} as follows.

\begin{corollary}
Let $b$ be a locally integrable function and $0<\beta<1$,
then the following assertions are equivalent:

(1)\ \ $b\in {\dot{\Lambda}_{\beta}(\rn)}$.

(2)\ \ $M_b$ is bounded from $L^{p(\cdot)}(\rn)$ to $L^{q(\cdot)}(\rn)$
for all $(p(\cdot),q(\cdot)) \in \mathscr{B}_{p,q}^{\beta}(\rn)$.

(3)\ \ $M_b$ is bounded from $L^{p(\cdot)}(\rn)$ to
$L^{q(\cdot)}(\rn)$  for some $(p(\cdot),q(\cdot)) \in
\mathscr{B}_{p,q}^{\beta}(\rn)$.

(4)$'$\  There exists some $q(\cdot)\in \mathscr{B}(\rn)$, such that
$$\sup_{Q} \frac1{|Q|^{\beta/n}}
\frac{\|(b-b_{Q})\chi_{Q}\|_{L^{q(\cdot)}(\rn)}}
{\|\chi_{Q}\|_{L^{q(\cdot)}(\rn)}}
<\infty.
$$

(5)$'$\ For all  $q(\cdot)\in \mathscr{B}(\rn)$, we have
$$\sup_{Q} \frac1{|Q|^{\beta/n}} \frac{\|(b- b_{Q})\chi_{Q}\|_{L^{q(\cdot)}(\rn)}}
{\|\chi_{Q}\|_{L^{q(\cdot)}(\rn)}} <\infty.
$$
\end{corollary}

This paper is organized as follows. In the next section, we recall
some basic definitions and known results. In Section 3, we will prove
Theorems \ref{thm.mc-lp} and \ref{thm.nc-lp}.
Section 4 is devoted to proving Theorem \ref{thm.nc-sharp}.
We will prove Theorems \ref{thm.equiv-norm-1} \--- \ref{thm.equiv-norm-3}
in the last section.

\section{Preliminaries and Lemmas}

It is known that the Lipschitz space $\dot{\Lambda}_{\beta}(\rn)$
coincides with some Morrey-Companato space and can be characterized
by mean oscillation. The following lemma is due to DeVore and Sharpley
\cite{ds} and Janson, Taibleson and Weiss \cite{jtw}
(see also Paluszy\'nski \cite{p}).

\begin{lemma}    \label{lem.lip}
Let $0<\beta<1$ and $1\le {q}<\infty$. Define
$$\dot{\Lambda}_{\beta,q}(\rn):=\bigg\{f \in{L_{\rm{loc}}^1(\rn)}:
\|f\|_{\dot{\Lambda}_{\beta,q}} = \sup_Q \frac1{|Q|^{\beta/n}}
\bigg(\frac1{|Q|} \int_Q|f(x)-f_Q|^q dx\bigg)^{1/q}<\infty \bigg\}.
$$
Then, for all $0<\beta<1$ and $1\le {q}<\infty$,
$\dot{\Lambda}_{\beta}(\rn)=\dot{\Lambda}_{\beta,q}(\rn)$
with equivalent norms.
\end{lemma}

Denoted by $p'(\cdot)$ the conjugate index of $p(\cdot)$.
It is easy to check that if $p(\cdot)\in \mathscr{P}(\rn)$
then $p'(\cdot)\in \mathscr{P}(\rn)$.
The following lemma is known as the generalized H\"older's
inequality in variable exponent Lebesgue spaces. See \cite{cf} and
\cite{dhhr} for details.

\begin{lemma} \label{lem.holder}
{(i)} Let $p(\cdot)\in \mathscr{P}(\rn)$. Then there exists a positive
constant $C$ such that for all $f\in L^{p(\cdot)}(\rn)$ and
$g\in L^{p'(\cdot)}(\rn)$,
$$\int_{\rn} |f(x)g(x)|dx \le {C} \|f\|_{L^{p(\cdot)}(\rn)}
\|g\|_{L^{p'(\cdot)}(\rn)}.
$$

{(ii)} Let $p(\cdot),~ p_{1}(\cdot),~ p_{2}(\cdot)\in \mathscr{P}(\rn)$
and $1/p(\cdot)=1/p_{1}(\cdot)+1/p_{2}(\cdot)$. Then there exists a
positive constant $C$ such that for all
$f\in {L^{p_{1}(\cdot)}}(\rn)$ and $g\in {L^{p_{2}(\cdot)}(\rn)}$,
$$ \|fg\|_{L^{p(\cdot)}(\rn)} \le {C}
 \|f\|_{L^{p_{1}(\cdot)}(\rn)}  \|g\|_{L^{p_{2}(\cdot)}(\rn)}.
$$
\end{lemma}

\begin{lemma}[\cite{cu-w}]  \label{lem.s-norm}
Given $p(\cdot) \in \mathscr{P}$, then for all $s>0$, we have
$$\big\||f|^s\big\|_{L^{p(\cdot)}(\rn)}=\|f\|^s_{L^{sp(\cdot)}(\rn)}.
$$
\end{lemma}

\begin{lemma}[\cite{i1}] \label{lem.cube}
Let $ q(\cdot)\in \mathscr{B}(\rn)$, then there exists a constant
$C>0$ such that
$$\frac{1}{|Q|} \|\chi_{Q}\|_{L^{q(\cdot)}(\rn)}
\|\chi_{Q}\|_{L^{q'(\cdot)}(\rn)} \le C
$$
for all cubes $Q$ in $\rn$.
\end{lemma}

\begin{lemma}  \label{lem.cube2}
Let $0<\alpha<n$. If $(p(\cdot),q(\cdot)) \in\mathscr{B}_{p,q}^{\alpha}(\rn)$,
then there exists a constant $C>0$ such that for all cubes $Q$,
$$\|\chi_Q\|_{L^{p(\cdot)}(\rn)}
\le {C}|Q|^{\alpha/n}\|\chi_Q\|_{L^{q(\cdot)}(\rn)}.
$$
\end{lemma}

\begin{proof}
Since $(p(\cdot),q(\cdot)) \in\mathscr{B}_{p,q}^{\alpha}(\rn)$
then $p(\cdot),q(\cdot) \in\mathscr{P}(\rn)$ and
$1/p(\cdot)=1/q(\cdot) + \beta/n$. The desired inequality
follows from Lemma \ref{lem.holder} (ii) directly.
\end{proof}

Let $0<\alpha<n$ and $f$ be a locally integrable function, the
fractional maximal function of $f$ is defined by
$$ \mathfrak{M}_{\alpha}(f)(x) =\sup_{Q\ni{x}}\frac1{|Q|^{1-\alpha/n}} \int_Q|f(y)|dy
$$
where the supremum is taken over all cubes $Q\subset\rn$ containing $x$.

The following result follows from Corollary 2.12 and Remark 2.13
of \cite{cfmp}, which improves the corresponding result in Capone,
Cruz-Uribe and Fiorenza \cite{ccf}.


\begin{lemma}[\cite{cfmp}] \label{lem.fractional}
Let $p(\cdot), q(\cdot)\in \mathscr{P}(\rn)$, $0<\alpha<n/p_{+}$
and $1/q(\cdot)=1/p(\cdot)-\alpha/n$. If $q(\cdot)(n-\alpha)/n \in
\mathscr{B}(\rn)$, then $\mathfrak{M}_{\alpha}$ is bounded from
$L^{p(\cdot)}(\rn)$ to $L^{q(\cdot)}(\rn)$.
\end{lemma}

\section{Proof of Theorems \ref{thm.mc-lp} and \ref{thm.nc-lp}}

\begin{proof} {\bf of Theorem \ref{thm.mc-lp}}~ Since the implications
$(2)\Rightarrow(3)$ and $(5)\Rightarrow(4)$ follow readily,
and $(2)\Rightarrow(5)$ is similar to $(3)\Rightarrow(4)$,
we only need to prove $(1)\Rightarrow(2)$, $(3)\Rightarrow(4)$
and $(4)\Rightarrow(1)$.

(1) $\Longrightarrow$ (2).
If $b\in {\dot{\Lambda}_{\beta}(\rn)}$, then
\begin{equation}        \label{equ.thm.mc-1}
\begin{split}
M_b(f)(x) & =\sup_{Q\ni {x}} \frac1{|Q|} \int_Q
|b(x)-b(y)||f(y)|dy\\
  &\le {C}\|b\|_{\dot{\Lambda}_{\beta}} \sup_{Q\ni {x}}
    \frac1{|Q|^{1-\beta/n}} \int_Q |f(y)|dy\\
    &= {C}\|b\|_{\dot{\Lambda}_{\beta}} \mathfrak{M}_{\beta}(f)(x).
\end{split}
\end{equation}
Obviously, assertion (2) follows from Lemma \ref{lem.fractional}
and (\ref{equ.thm.mc-1}).

(3) $\Longrightarrow$ (4).
For any fixed cube $Q$, we have for all $x\in{Q}$,
\begin{equation*}
\begin{split}
|b(x)-b_Q|
 &\le \frac1{|Q|}\int_Q|b(x)-b(y)|dy\\
 &=\frac1{|Q|}\int_Q|b(x)-b(y)|\chi_Q(y)dy\\
 &\le {M_b(\chi_Q)(x)}.
\end{split}
\end{equation*}
Then, for all $x\in \rn$,
$$|(b(x)-b_Q)\chi_Q(x)| \le {M_b(\chi_Q)(x)}.
$$

By assertion (3), there exists
$(p(\cdot),q(\cdot)) \in\mathscr{B}_{p,q}^{\beta}(\rn)$
such that $M_b$ is bounded from $L^{p(\cdot)}(\rn)$ to
$L^{q(\cdot)}(\rn)$. For this $(p(\cdot),q(\cdot))$,
it follows from assertion (3) and Lemma \ref{lem.cube2}
that there exists a constant $C>0$, independent of $Q$,
such that
\begin{equation}        \label{equ.thm.mc-2}
\begin{split}
\|(b-b_Q)\chi_Q\|_{L^{q(\cdot)}(\rn)}
 &\le \|M_b(\chi_Q)\|_{L^{q(\cdot)}(\rn)}\\
  &\le {C}\|M_b\|_{L^{p(\cdot)}\to {L^{q(\cdot)}}}
   \|\chi_Q\|_{L^{p(\cdot)}(\rn)}\\
  &\le {C}\|M_b\|_{L^{p(\cdot)}\to {L^{q(\cdot)}}}
    |Q|^{\beta/n}\|\chi_Q\|_{L^{q(\cdot)}(\rn)},
\end{split}
\end{equation}
which implies assertion (4) since $Q$ is arbitrary
and $C$ is independent of $Q$.

(4) $\Longrightarrow$ (1). For any cube $Q$,
by Lemma \ref{lem.holder} (i), assertion (4) and Lemma \ref{lem.cube},
we have
\begin{equation}         \label{equ.thm.mc-3}
\begin{split}
&\frac1{|Q|^{1+\beta/n}}\int_Q|b(x)-b_Q|dx\\
&=\frac1{|Q|^{1+\beta/n}}\int_Q|b(x)-b_Q|\chi_Q(x)dx\\
 &\le \frac{C}{|Q|^{1+\beta/n}}
  \|(b-b_Q)\chi_Q\|_{L^{q(\cdot)}(\rn)} \|\chi_Q\|_{L^{q'(\cdot)}(\rn)}\\
 & =\frac{C}{|Q|^{\beta/n}}
    \frac{\|(b-b_Q)\chi_Q\|_{L^{q(\cdot)}(\rn)}}{\|\chi_Q\|_{L^{q(\cdot)}(\rn)}}
     \frac1{|Q|} \|\chi_Q\|_{L^{q(\cdot)}(\rn)} \|\chi_Q\|_{L^{q'(\cdot)}(\rn)}\\
   &\le \frac{C}{|Q|^{\beta/n}}
    \frac{\|(b-b_Q)\chi_Q\|_{L^{q(\cdot)}(\rn)}}{\|\chi_Q\|_{L^{q(\cdot)}(\rn)}}\\
    & \le {C}.
\end{split}
\end{equation}
This shows that $b\in {\dot{\Lambda}_{\beta}(\rn)}$ by Lemma \ref{lem.lip}
since the constant $C$ is independent of $Q$.

The proof of Theorem \ref{thm.mc-lp} is finished.
\end{proof}

From the proof of Theorem \ref{thm.mc-lp}, we have the
following Corollary.

\begin{corollary}        \label{cor-thm.mc-lp}
{(i)} Let $0<\beta<1$ and $q(\cdot) \in {\mathscr{B}(\rn)}$.
If $b \in {\dot{\Lambda}_{\beta,q(\cdot)}(\rn)}$ then
$b \in {\dot{\Lambda}_{\beta}(\rn)}$ and
$\|b\|_{\dot{\Lambda}_{\beta}} \le {C}
\|b\|_{\dot{\Lambda}_{\beta,q(\cdot)}}$ for some
positive constant $C$.

{(ii)} Let $0<\beta<1$ and $q(\cdot)$ satisfy that there exists
$p(\cdot)$ such that $(p(\cdot),q(\cdot)) \in\mathscr{B}_{p,q}^{\beta}(\rn)$.
If $b \in {\dot{\Lambda}_{\beta}(\rn)}$ then
$b \in {\dot{\Lambda}_{\beta,q(\cdot)}(\rn)}$ and
$\|b\|_{\dot{\Lambda}_{\beta,q(\cdot)}}
\le{C}\|b\|_{\dot{\Lambda}_{\beta}}$ for some
positive constant $C$.
\end{corollary}

\begin{proof} (i) For any $q(\cdot) \in {\mathscr{B}(\rn)}$,
if $b \in {\dot{\Lambda}_{\beta,q(\cdot)}(\rn)}$ then
(\ref{equ.1.2}) holds by Definition \ref{def.1.4}, which
implies (\ref{equ.thm.mc-3}). Then the desired result follows.

(ii) If $b \in {\dot{\Lambda}_{\beta}(\rn)}$ then the wanted
result follows from (\ref{equ.thm.mc-1}), Lemma \ref{lem.fractional}
and (\ref{equ.thm.mc-2}).
\end{proof}

\begin{proof} {\bf of Theorem \ref{thm.nc-lp}}~ Similar to the
proof of Theorem \ref{thm.mc-lp}, it is sufficient to prove
$(1)\Rightarrow(2)$, $(3)\Rightarrow(4)$ and $(4)\Rightarrow(1)$.

(1) $\Longrightarrow$ (2).
Let $b\in {\dot{\Lambda}_{\beta}(\rn)}$ and $b\ge 0$.
For any $x\in \rn$ such that $M(f)(x)<\infty$,
we have
\begin{equation}        \label{equ.thm.nc-0}
\begin{split}
|[b,M](f)(x)| & =\bigg| \sup_{Q\ni{x}} \frac1{|Q|}\int_Q b(x)f(y)dy
 - \sup_{Q\ni{x}} \frac1{|Q|} \int_Q b(y)f(y)dy \bigg| \\
  &\le \sup_{Q\ni{x}} \frac1{|Q|}\int_Q|b(x)-b(y)||f(y)|dy\\
  &\le {C}\|b\|_{\dot{\Lambda}_{\beta}} \sup_{Q\ni {x}}
    \frac1{|Q|^{1-\beta/n}} \int_Q |f(y)|dy\\
    &= {C}\|b\|_{\dot{\Lambda}_{\beta}} \mathfrak{M}_{\beta}(f)(x).
\end{split}
\end{equation}
Then, by Lemma \ref{lem.fractional}, the implication $(1) \Rightarrow (2)$
is proven.

(3) $\Longrightarrow$ (4). For any fixed cube $Q$, noting
that for all $x\in {Q}$, we have (see the proof of Proposition 4.1
in \cite{bmr}, see also (2.4) in \cite{zhw1})
$$M(\chi_Q)(x)=\chi_Q(x) ~~\hbox{and}~~ M(b\chi_Q)(x)=M_Q(b)(x).
$$

By assertion (3), there exists $(p(\cdot),q(\cdot)) \in
\mathscr{B}_{p,q}^{\beta}(\rn)$ such that $[b,M]$ is bounded
from $L^{p(\cdot)}(\rn)$ to $L^{q(\cdot)}(\rn)$. Then
\begin{equation*}
\begin{split}
\|(b-M_{Q}(b))\chi_{Q}\|_{L^{q(\cdot)}(\rn)}
 &=\|(bM(\chi_Q)-M(b\chi_Q))\chi_Q\|_{L^{q(\cdot)}(\rn)}\\
 &\le \|bM(\chi_Q)-M(b\chi_Q) \|_{L^{q(\cdot)}(\rn)}\\
 &= \|[b,M](\chi_Q) \|_{L^{q(\cdot)}(\rn)}\\
  &\le {C} \|[b,M]\|_{L^{p(\cdot)}\to {L^{q(\cdot)}}}
    \|\chi_Q\|_{L^{p(\cdot)}(\rn)}.
\end{split}
\end{equation*}
This together with Lemma \ref{lem.cube2} gives that, for any cube $Q$,
$$\frac1{|Q|^{\beta/n}}
\frac{\|(b-M_{Q}(b)\big)\chi_{Q}\|_{L^{q(\cdot)}(\rn)}}{\|\chi_{Q}\|_{L^{q(\cdot)}(\rn)}}
    \le {C} \|[b,M]\|_{L^{p(\cdot)}\to {L^{q(\cdot)}}},
$$
where the constant $C$ is independent of $Q$. Thus, the proof of
(3) $\Longrightarrow$ (4) is completed.

(4) $\Longrightarrow$ (1).  We first prove
$b\in{\dot{\Lambda}_{\beta}(\rn)}$.
For any fixed cube $Q$, by using similar procedure to the proof
of ``(4.4) $\Rightarrow$ (4.3)'' in \cite{bmr},
we can obtain
\begin{equation}              \label{equ.thm.nc-1} 
\frac1{|Q|^{1+\beta/n}} \int_Q |b(x)-b_Q|dx
 \le \frac2{|Q|^{1+\beta/n}}\int_Q|b(x)-M_Q(b)(x)|dx.
\end{equation}

We give the proof of (\ref{equ.thm.nc-1}) for completeness.
Let $E=\{x\in {Q}: b(x)\le {b_Q}\}$.
The following equality is true (see \cite{bmr} page 3331):
$$\int_E|b(x)-b_Q|dx=\int_{Q\setminus{E}} |b(x)-b_Q|dx.
$$

Since for any $x\in{E}$ we have $b(x)\le {b_Q} \le {M_Q(b)(x)}$, then
for any  $x\in{E}$,
$$|b(x)-b_Q| \le |b(x)-M_Q(b)(x)|.
$$
Thus,
\begin{equation*}
\begin{split}
\frac1{|Q|^{1+\beta/n}} \int_Q |b(x)-b_Q|dx
  &=\frac1{|Q|^{1+\beta/n}}\int_{E\cup{(Q\setminus{E})}}|b(x)-b_Q|dx\\
  &=\frac2{|Q|^{1+\beta/n}}\int_E|b(x)-b_Q|dx\\
    &\le \frac2{|Q|^{1+\beta/n}}\int_E|b(x)-M_Q(b)(x)|dx\\
        &\le \frac2{|Q|^{1+\beta/n}}\int_Q|b(x)-M_Q(b)(x)|dx.
\end{split}
\end{equation*}

On the other hand, it follows from Lemma \ref{lem.holder} (i),
assertion (4) and Lemma \ref{lem.cube} that
\begin{equation}    \label{equ.thm.nc-2}  
\begin{split}
&\frac1{|Q|^{1+\beta/n}}\int_Q|b(x)-M_Q(b)(x)|dx\\
 &=\frac1{|Q|^{1+\beta/n}}\int_Q|b(x)-M_Q(b)(x)|\chi_Q(x)dx \\
 &\le \frac{C}{|Q|^{1+\beta/n}}\|(b-M_Q(b))\chi_Q\|_{L^{q(\cdot)}(\rn)}
    \| \chi_Q\|_{L^{q'(\cdot)}(\rn)} \\
  &\le \frac{C}{|Q|^{\beta/n}}
    \frac{\|(b-M_Q(b))\chi_Q\|_{L^{q(\cdot)}(\rn)}}{\|\chi_Q\|_{L^{q(\cdot)}(\rn)}}
    \frac{1}{|Q|}\|\chi_Q\|_{L^{q(\cdot)}(\rn)} \chi_Q\|_{L^{q'(\cdot)}(\rn)} \\
    &\le {C},
\end{split}
\end{equation}
this together with (\ref{equ.thm.nc-1}) gives $b\in{\dot{\Lambda}_{\beta}(\rn)}$.

Now, we prove $b\ge 0$. To do this, it suffices
to show $b^{-}=0$, where $b^{-}=-\min\{b,0\}$.
Let $b^{+}=|b|-b^{-}$, then $b=b^{+}-b^{-}$.
For any fixed cube $Q$,
$$0\le {b^{+}(x)}\le |b(x)| \le {M_Q(b)(x)}, ~~~x\in{Q}.
$$
Therefore, for $x\in{Q}$, we have
$$0\le {b^{-}(x)}\le {M_Q(b)(x)}- b^{+}(x) + {b^{-}(x)}
= M_Q(b)(x)- b(x).
$$

Then, it follows from (\ref{equ.thm.nc-2}) that, for any cube $Q$,
\begin{align*}
\frac1{|Q|}\int_Q b^{-}(x)dx \le \frac1{|Q|}\int_Q |M_Q(b)(x)- b(x)|dx
\le {C}|Q|^{\beta/n}.
\end{align*}
Thus, $b^{-}=0$ follows from Lebesgue's differentiation theorem.

The proof of Theorem \ref{thm.nc-lp} is completed.
\end{proof}

Similar to Corollary \ref{cor-thm.mc-lp},
from the proof of Theorem \ref{thm.nc-lp}, we
obtain the following result.

\begin{corollary}        \label{cor-thm.nc-lp}
(i) Let $0<\beta<1$ and $q(\cdot) \in {\mathscr{B}(\rn)}$.
If $b \in {\dot{\Lambda}_{\beta,q(\cdot)}^{\ast}(\rn)}$ then
$b \in {\dot{\Lambda}_{\beta}^{+}(\rn)}$ and
$\|b\|_{\dot{\Lambda}_{\beta}} \le {C}
\|b\|_{\dot{\Lambda}^{\ast}_{\beta,q(\cdot)}}$ for some
positive constant $C$.

(ii) Let $0<\beta<1$ and $q(\cdot)$ satisfy that there exists
$p(\cdot)$ such that $(p(\cdot),q(\cdot)) \in\mathscr{B}_{p,q}^{\beta}(\rn)$.
If $b \in {\dot{\Lambda}^{+}_{\beta}(\rn)}$ then
$b \in {\dot{\Lambda}^{\ast}_{\beta,q(\cdot)}(\rn)}$ and
$\|b\|_{\dot{\Lambda}^{\ast}_{\beta,q(\cdot)}}
\le{C}\|b\|_{\dot{\Lambda}_{\beta}}$ for some
positive constant $C$.
\end{corollary}

\section{Proof of Theorem \ref{thm.nc-sharp}}

\begin{proof}{\bf of Theorem \ref{thm.nc-sharp}} ~We only need to prove
(1) $\Rightarrow$ (2), (3) $\Rightarrow$ (4) and (4) $\Rightarrow$ (1).

(1) $\Longrightarrow$ (2).~ Since $b\in{\dot{\Lambda}_{\beta}(\rn)}$
and $b\ge 0$, then for any fixed $x\in\rn$ such that $Mf(x)<\infty$,
\begin{equation*}
\begin{split}
|[b,M^{\sharp}]f(x)| &=\bigg|\sup_{Q\ni{x}}\frac{b(x)}{|Q|}
   \int_Q\big|f(y)-f_Q\big|dy
     -\sup_{Q\ni{x}}\frac{1}{|Q|}\int_Q\big|b(y)f(y)-(bf)_Q\big|dy \bigg|\\
&\le\sup_{Q\ni{x}}\frac{1}{|Q|}\int_Q \big|\big(b(y)-b(x)\big)f(y)
 +b(x)f_Q -(bf)_Q\big|dy\\
&\le\sup_{Q\ni{x}}\bigg\{\frac{1}{|Q|}\int_Q |b(y)-b(x)||f(y)|dy
 +\big|b(x)f_Q -(bf)_Q\big|\bigg\}\\
&\le {C}\|b\|_{\dot{\Lambda}_{\beta}} \mathfrak{M}_{\beta}(f)(x)
 +\sup_{Q\ni{x}} \bigg|\frac{b(x)}{|Q|}\int_Qf(z)dz
 -\frac{1}{|Q|}\int_Qb(z)f(z)dz\bigg|\\
&\le {C}\|b\|_{\dot{\Lambda}_{\beta}} \mathfrak{M}_{\beta}(f)(x)
 +\sup_{Q\ni{x}}\frac{1}{|Q|}\int_Q|b(x)-b(z)||f(z)|dz\\
&\le{C}\|b\|_{\dot{\Lambda}_{\beta}} \mathfrak{M}_{\beta}(f)(x).
\end{split}
\end{equation*}
By Lemma \ref{lem.fractional}, $[b,M^{\sharp}]$ is
bounded from $L^{p(\cdot)}(\rn)$ to $L^{q(\cdot)}(\rn)$ for all
$(p(\cdot),q(\cdot)) \in \mathscr{B}_{p,q}^{\beta}(\rn)$.

(3) $\Longrightarrow$ (4). For any fixed cube $Q$, we have
(see \cite{bmr} page 3333 or \cite{zhw3} page 1383 for details),
$$M^{\sharp}(\chi_Q)(x)=1/2, ~~\hbox{for~all~}x \in {Q}.
$$
Then, by assertion (3) and Lemma \ref{lem.cube2}, we have
\begin{equation*}
\begin{split}
\big\|\big(b- 2M^{\sharp}(b\chi_{Q})\big)\chi_Q
  \big\|_{L^{q(\cdot)}(\rn)}
&=\bigg\|2\Big(\frac{1}{2}b -M^{\sharp}(b\chi_Q)\Big)\chi_Q
  \bigg\|_{L^{q(\cdot)}(\rn)} \\
&=\big\|2\big(bM^{\sharp}(\chi_Q)
  -M^{\sharp}(b\chi_{Q})\big)\chi_Q \big\|_{L^{q(\cdot)}(\rn)}\\
&\le 2\big\|[b,M^{\sharp}](\chi_{Q})\big\|_{L^{q(\cdot)}(\rn)}\\
&\le {C} \big\|[b,M^{\sharp}]\big\|_{L^{p(\cdot)}\to{L^{q(\cdot)}}}
  \|\chi_{Q}\|_{L^{p(\cdot)}(\rn)}\\
& \le {C}\big\|[b,M^{\sharp}]\big\|_{L^{p(\cdot)}\to{L^{q(\cdot)}}}
     |Q|^{\beta/n} \|\chi_Q\|_{L^{q(\cdot)}(\rn)},
\end{split}
\end{equation*}
where the constant $C$ is independent of $Q$. Then
$$\sup_{Q}\frac1{|Q|^{\beta/n}}
\frac{\|(b-2M^{\sharp}(b\chi_{Q}))\chi_Q\|_{L^{q(\cdot)}(\rn)}}
  {\|\chi_Q\|_{L^{q(\cdot)}(\rn)}}
     \le{C}\big\|[b,M^{\sharp}]\big\|_{L^{p(\cdot)}\to{L^{q(\cdot)}}}.
$$

(4) $\Longrightarrow$ (1). We first prove $b\in{\dot{\Lambda}_{\beta}(\rn)}$.
For any cube $Q\subset \rn$, we have (see (2) in \cite{bmr}):
\begin{equation}        \label{equ.csharp-1}
\begin{split}
|b_Q|\le 2M^{\sharp}(b\chi_Q)(x)~~ \hbox{for}~~ x\in {Q}.
\end{split}
\end{equation}

Let $E=\{x\in Q: b(x)\le b_Q\}$, then
$$\int_E |b(x)-b_Q|dx =\int_{Q\setminus{E}} |b(x)-b_Q|dx.
$$

Since for any $ x\in E$, we have $b(x)\le b_Q\le |b_Q| \le 2
M^{\sharp}(b\chi_Q)(x)$, then
$$|b(x)- b_Q|\le |b(x)-2  M^{\sharp}(b\chi_Q)(x)|,
~~\hbox{for}~~ x\in {E}.
$$

By Lemma \ref{lem.holder} (i), assertion (4) and Lemma \ref{lem.cube},
we obtain
\begin{equation*} 
\begin{split}
\frac{1}{|Q|^{1+\beta/n}} \int_{Q}|b(x)-b_Q|dx
  &= \frac{2}{|Q|^{1+\beta/n}}\int_E |b(x)-b_Q|dx\\
&\le \frac{2}{|Q|^{1+\beta/n}}\int_{E}
 |b(x)-2M^{\sharp}(b\chi_Q)(x)|dx\\
&\le \frac{2}{|Q|^{1+\beta/n}} \int_{Q}
  |b(x)-2M^{\sharp}(b\chi_Q)(x)|dx\\
&\le \frac{C}{|Q|^{1+\beta/n}} \big\|(b-2M^{\sharp}(b\chi_Q))
  \chi_Q \big\|_{L^{q(\cdot)}(\rn)}
  \|\chi_{Q}\|_{L^{q'(\cdot)}(\rn)} \\
  &\le \frac{C}{|Q|^{\beta/n}}
  \frac{\big\|(b-2M^{\sharp}(b\chi_Q))\chi_Q \big\|_{L^{q(\cdot)}(\rn)}}
  { \|\chi_{Q}\|_{L^{q(\cdot)}(\rn)}}\\
  &\qquad \qquad \times
     \frac{1}{|Q|} \|\chi_{Q}\|_{L^{q(\cdot)}(\rn)}
         \|\chi_{Q}\|_{L^{q'(\cdot)}(\rn)}\\
 &  \le {C},
\end{split}
\end{equation*}
which implies $b\in{\dot{\Lambda}_{\beta}(\rn)}$ by Lemma \ref{lem.lip}.

Now, let us prove $b\ge 0$. It also suffices to show $b^{-}=0$,
where $b^{-}=-\min\{b,0\}$ and let $b^{+}=|b|-b^{-}$.
By (\ref{equ.csharp-1}) we have, for $x\in Q$,
$$2M^{\sharp}(b\chi_Q)(x)-b(x) \ge |b_Q|-b(x)
 = |b_Q|-b^{+}(x)+b^{-}(x).
$$
Then
\begin{equation}         \label{equ.csharp-3}  
\begin{split}
\frac{1}{|Q|} \int_{Q} \big|2 M^{\sharp} (b\chi_Q)(x)-b(x)\big|dx
&\ge \frac{1}{|Q|}\int_{Q} \big(2 M^{\sharp}(b\chi_Q)(x)-b(x)\big)dx \\
&\ge \frac{1}{|Q|}\int_{Q} \big(|b_Q|
  - b^{+}(x)+b^{-}(x)\big)dx  \\
&= |b_Q| -\frac{1}{|Q|}\int_{Q} b^{+}(x)dx
  + \frac{1}{|Q|}\int_{Q}  b^{-}(x)dx.
\end{split}
\end{equation}

On the other hand, applying Lemma \ref{lem.holder} (i),
assertion (4) and Lemma \ref{lem.cube}, we have
\begin{equation*}
\begin{split}
&\frac{1}{|Q|}\int_{Q}\big|2M^{\sharp}(b\chi_Q)(x)-b(x)\big|dx\\
&\le \frac{C}{|Q|}\big\|\big(2M^{\sharp}(b\chi_Q)(x)
   -b(x)\big)\chi_Q\big\|_{L^{q(\cdot)}(\rn)}
   \|\chi_Q\|_{L^{q'(\cdot)}(\rn)}\\
&\le {C}|Q|^{-1}|Q|^{\beta/n} \|\chi_Q\|_{L^{q(\cdot)}(\rn)}
   \|\chi_Q\|_{L^{q'(\cdot)}(\rn)}\\
   & \le {C}|Q|^{\beta/n},
\end{split}
\end{equation*}
where the constant $C$ is independent of $Q$.
This, together with (\ref{equ.csharp-3}), gives
\begin{equation}         \label{equ.csharp-4}
|b_Q| -\frac{1}{|Q|}\int_{Q} b^{+}(x)dx
  + \frac{1}{|Q|}\int_{Q}  b^{-}(x)dx \le {C}|Q|^{\beta/n}.
\end{equation}

Let the side length of $Q$ tends to $0$ (then $|Q|\to 0$)
with $x\in Q$, Lebesgue's differentiation theorem assures
that the limit of the left-hand side of (\ref{equ.csharp-4})
equals to
$$|b(x)|-b^{+}(x)+b^{-}(x)=2b^{-}(x)=2|b^{-}(x)|.
$$
And, the right-hand side of (\ref{equ.csharp-4}) tends to $0$.
So, we have $b^{-}=0$.

The proof of Theorem \ref{thm.nc-sharp} is completed.
\end{proof}

Similar to Corollary \ref{cor-thm.mc-lp},
from the proof of Theorem \ref{thm.nc-sharp}, we
obtain the following result.

\begin{corollary}        \label{cor-thm.nc-sharp}
(i) Let $0<\beta<1$ and $q(\cdot) \in {\mathscr{B}(\rn)}$.
If $b \in {\dot{\Lambda}^{\sharp}_{\beta,q(\cdot)}(\rn)}$ then
$b \in {\dot{\Lambda}_{\beta}^{+}(\rn)}$ and
$\|b\|_{\dot{\Lambda}_{\beta}} \le {C}
\|b\|_{\dot{\Lambda}^{\sharp}_{\beta,q(\cdot)}}$
for some positive constant $C$.

(ii) Let $0<\beta<1$ and $q(\cdot)$ satisfy that there exists
$p(\cdot)$ such that $(p(\cdot),q(\cdot)) \in\mathscr{B}_{p,q}^{\beta}(\rn)$.
If $b \in {\dot{\Lambda}^{+}_{\beta}(\rn)}$ then
$b \in {\dot{\Lambda}^{\sharp}_{\beta,q(\cdot)}(\rn)}$ and
$\|b\|_{\dot{\Lambda}^{\sharp}_{\beta,q(\cdot)}}
\le{C}\|b\|_{\dot{\Lambda}_{\beta}}$
for some positive constant $C$.
\end{corollary}

\section{Proof of Theorems \ref{thm.equiv-norm-1} \--- \ref{thm.equiv-norm-3}}

\begin{proof}{\bf of Theorem \ref{thm.equiv-norm-1}} ~
By Corollary \ref{cor-thm.mc-lp}, the only thing we need to do
is to prove that $b\in {\dot{\Lambda}_{\beta}}(\rn)$ implies
$b\in {\dot{\Lambda}_{\beta,q(\cdot)}(\rn)}$ and
$\|b\|_{\dot{\Lambda}_{\beta,q(\cdot)}}\le{C}\|b\|_{\dot{\Lambda}_{\beta}}$
for all $q(\cdot) \in\mathscr{B}(\rn)$.

For any fixed $q(\cdot) \in\mathscr{B}(\rn)$, choose
$r>n/(n-\beta)$. By Remark \ref{remark1},
$rq(\cdot)(n-\beta)/n \in\mathscr{B}(\rn)$ and
$rq(\cdot)\in\mathscr{B}(\rn)$ since $r>n/(n-\beta)$.
Set $q_0(\cdot)=rq(\cdot)$ and define $p_0(\cdot)$ by
$$\frac1{p_0(x)}=\frac1{q_0(x)}+\frac{\beta}{n}.
$$

It is easy to check that
$(p_0(\cdot),q_0(\cdot)) \in{\mathscr{B}}^{\beta}_{p,q}(\rn)$.
Then, for any $b\in {\dot{\Lambda}_{\beta}}(\rn)$,
by Corollary \ref{cor-thm.mc-lp} we have
$b\in {\dot{\Lambda}_{\beta,q_0(\cdot)}(\rn)}$ and
$\|b\|_{\dot{\Lambda}_{\beta,q_0(\cdot)}}\le{C}\|b\|_{\dot{\Lambda}_{\beta}}$.

For any fixed cube $Q$, since
$$\frac1{q(\cdot)}=\frac1{rq(\cdot)}+\frac1{r'q(\cdot)}
=\frac1{q_0(\cdot)}+\frac1{r'q(\cdot)},
$$
then, it follows from Lemma \ref{lem.holder} (ii),
Lemma \ref{lem.s-norm} and
$b\in {\dot{\Lambda}_{\beta,q_0(\cdot)}(\rn)}$ that
\begin{align*}
\frac1{|Q|^{\beta/n}}\frac{\|(b-b_Q)\chi_Q\|_{L^{q(\cdot)}(\rn)}}
  {\|\chi_Q\|_{L^{q(\cdot)}(\rn)}}
&\le \frac{C\|(b-b_Q)\chi_Q\|_{L^{q_0(\cdot)}(\rn)}
  \|\chi_Q\|_{L^{r'q(\cdot)}(\rn)}}{|Q|^{\beta/n}\|\chi_Q\|_{L^{q(\cdot)}(\rn)}}\\
 & \le{C}\|b\|_{\dot{\Lambda}_{\beta,q_0(\cdot)}}
  \frac{\|\chi_Q\|_{L^{q_0(\cdot)}(\rn)}
   \|\chi_Q\|_{L^{r'q(\cdot)}(\rn)}}{\|\chi_Q\|_{L^{q(\cdot)}(\rn)}}\\
 & \le{C}\|b\|_{\dot{\Lambda}_{\beta}}
  \frac{\|\chi_Q\|^{1/r}_{L^{q(\cdot)}(\rn)}
   \|\chi_Q\|^{1/r'}_{L^{q(\cdot)}(\rn)}}{\|\chi_Q\|_{L^{q(\cdot)}(\rn)}}\\
 & \le{C}\|b\|_{\dot{\Lambda}_{\beta}},
 \end{align*}
which implies
$b\in {\dot{\Lambda}_{\beta,q(\cdot)}(\rn)}$ and
$\|b\|_{\dot{\Lambda}_{\beta,q(\cdot)}}\le{C}\|b\|_{\dot{\Lambda}_{\beta}}$
for all $q(\cdot) \in\mathscr{B}(\rn)$.
\end{proof}

\begin{proof}{\bf of Theorem \ref{thm.equiv-norm-2}} ~
By Corollary \ref{cor-thm.nc-lp}, it suffices to prove that
for any $b\in {\dot{\Lambda}^{+}_{\beta}}(\rn)$ we have
$b\in {\dot{\Lambda}^{\ast}_{\beta,q(\cdot)}(\rn)}$ and
$\|b\|_{\dot{\Lambda}^{\ast}_{\beta,q(\cdot)}}
\le{C}\|b\|_{\dot{\Lambda}_{\beta}}$
for all $q(\cdot) \in\mathscr{B}(\rn)$.

For any $q(\cdot) \in\mathscr{B}(\rn)$, let $r$,
$q_0(\cdot)$ and $p_0(\cdot)$ be the same as in
the proof of Theorem \ref{thm.equiv-norm-1}.
Then, for any $b\in {\dot{\Lambda}^{+}_{\beta}}(\rn)$,
by Corollary \ref{cor-thm.nc-lp} we have
$b\in {\dot{\Lambda}^{\ast}_{\beta,q_0(\cdot)}(\rn)}$ and
$\|b\|_{\dot{\Lambda}^{\ast}_{\beta,q_0(\cdot)}}\le{C}\|b\|_{\dot{\Lambda}_{\beta}}$.

For any fixed cube $Q$, note that
$b\in {\dot{\Lambda}^{\ast}_{\beta,q_0(\cdot)}(\rn)}$,
by Lemma \ref{lem.holder} (ii) and Lemma \ref{lem.s-norm},
\begin{equation*}
\begin{split}
\frac1{|Q|^{\beta/n}}\frac{\|(b-M_Q(b))\chi_Q\|_{L^{q(\cdot)}(\rn)}}
  {\|\chi_Q\|_{L^{q(\cdot)}(\rn)}}
&\le \frac{C\|(b-M_Q(b))\chi_Q\|_{L^{q_0(\cdot)}(\rn)}
  \|\chi_Q\|_{L^{r'q(\cdot)}(\rn)}}{|Q|^{\beta/n}\|\chi_Q\|_{L^{q(\cdot)}(\rn)}}\\
 & \le{C}\|b\|_{\dot{\Lambda}^{\ast}_{\beta,q_0(\cdot)}}
  \frac{\|\chi_Q\|_{L^{q_0(\cdot)}(\rn)}
   \|\chi_Q\|_{L^{r'q(\cdot)}(\rn)}}{\|\chi_Q\|_{L^{q(\cdot)}(\rn)}}\\
 & \le{C}\|b\|_{\dot{\Lambda}_{\beta}}
  \frac{\|\chi_Q\|^{1/r}_{L^{q(\cdot)}(\rn)}
   \|\chi_Q\|^{1/r'}_{L^{q(\cdot)}(\rn)}}{\|\chi_Q\|_{L^{q(\cdot)}(\rn)}}\\
 & \le{C}\|b\|_{\dot{\Lambda}_{\beta}},
\end{split}
\end{equation*}
which shows
$b\in {\dot{\Lambda}^{\ast}_{\beta,q(\cdot)}(\rn)}$ and
$\|b\|_{\dot{\Lambda}^{\ast}_{\beta,q(\cdot)}}\le{C}\|b\|_{\dot{\Lambda}_{\beta}}$
for all $q(\cdot) \in\mathscr{B}(\rn)$.
\end{proof}

\begin{proof}{\bf of Theorem \ref{thm.equiv-norm-3}} ~
By Corollary \ref{cor-thm.nc-sharp}, it is enough to prove
that for any $b\in {\dot{\Lambda}^{+}_{\beta}}(\rn)$ we have
$b\in {\dot{\Lambda}^{\sharp}_{\beta,q(\cdot)}(\rn)}$ and
$\|b\|_{\dot{\Lambda}^{\sharp}_{\beta,q(\cdot)}}\le{C}\|b\|_{\dot{\Lambda}_{\beta}}$
for all $q(\cdot) \in\mathscr{B}(\rn)$.

For any $q(\cdot) \in\mathscr{B}(\rn)$, let $r$,
$q_0(\cdot)$ and $p_0(\cdot)$ be the same as in
the proof of Theorem \ref{thm.equiv-norm-1}.
Then, for any $b\in {\dot{\Lambda}^{+}_{\beta}}(\rn)$,
by Corollary \ref{cor-thm.nc-sharp} we have
$b\in {\dot{\Lambda}^{\sharp}_{\beta,q_0(\cdot)}(\rn)}$ and
$\|b\|_{\dot{\Lambda}^{\sharp}_{\beta,q_0(\cdot)}}\le{C}\|b\|_{\dot{\Lambda}_{\beta}}$.

For any fixed cube $Q$, note that
$b\in {\dot{\Lambda}^{\ast}_{\beta,q_0(\cdot)}(\rn)}$,
by Lemma \ref{lem.holder} (ii) and Lemma \ref{lem.s-norm},
\begin{equation*}
\begin{split}
\frac1{|Q|^{\beta/n}}\frac{\|(b-2M^{\sharp}(b\chi_Q))\chi_Q\|_{L^{q(\cdot)}(\rn)}}
  {\|\chi_Q\|_{L^{q(\cdot)}(\rn)}}
&\le \frac{C\|(b-2M^{\sharp}(b\chi_Q))\chi_Q\|_{L^{q_0(\cdot)}(\rn)}
  \|\chi_Q\|_{L^{r'q(\cdot)}(\rn)}}{|Q|^{\beta/n}\|\chi_Q\|_{L^{q(\cdot)}(\rn)}}\\
 & \le{C}\|b\|_{\dot{\Lambda}^{\sharp}_{\beta,q_0(\cdot)}}
  \frac{\|\chi_Q\|_{L^{q_0(\cdot)}(\rn)}
   \|\chi_Q\|_{L^{r'q(\cdot)}(\rn)}}{\|\chi_Q\|_{L^{q(\cdot)}(\rn)}}\\
 & \le{C}\|b\|_{\dot{\Lambda}_{\beta}}
  \frac{\|\chi_Q\|^{1/r}_{L^{q(\cdot)}(\rn)}
   \|\chi_Q\|^{1/r'}_{L^{q(\cdot)}(\rn)}}{\|\chi_Q\|_{L^{q(\cdot)}(\rn)}}\\
 & \le{C}\|b\|_{\dot{\Lambda}_{\beta}},
\end{split}
\end{equation*}
which shows
$b\in {\dot{\Lambda}^{\sharp}_{\beta,q(\cdot)}(\rn)}$ and
$\|b\|_{\dot{\Lambda}^{\sharp}_{\beta,q(\cdot)}}\le{C}\|b\|_{\dot{\Lambda}_{\beta}}$
for all $q(\cdot) \in\mathscr{B}(\rn)$.
\end{proof}



\end{document}